\documentclass[letterpaper, 10 pt, conference]{ieeeconf} 
\usepackage{amsmath, amsfonts, amssymb, graphicx, cite, epsfig,epstopdf,color,soul,tikz}
\usepackage{url}
\usepackage[ruled]{algorithm2e}

\IEEEoverridecommandlockouts     
\newtheorem{lemma}{Lemma}
\newtheorem{theorem}{Theorem}
\newtheorem{assumption}{Assumption}

\newtheorem{proposition}{Proposition}

\newtheorem{notation}{\bf{Notation}}

\newcommand{\bld}{\boldsymbol{\lambda}}
\newcommand{\xbf}{\mathbf{x}}

\newcommand{\Xcal}{\mathcal{X}}
\newcommand{\Rset}{\mathbb{R}}
\newcommand{\Psf}{\sf{P}}

\newcommand{\Lcal}{\mathcal{L}}

\title{\LARGE \bf
Distributed Lagrangian Methods for Network Resource Allocation 
}

\author{Thinh T. Doan$^{1}$ and Carolyn L. Beck$^{2}$
\thanks{$^{1}$Thinh T. Doan is with the Deparment of Electrical and Computer Engineering, University of Illinois,  Urbana, IL, USA
        {\tt\small ttdoan2@illinois.edu}}%
\thanks{$^{2}$Carolyn Beck is with the Department of Industrial and Enterprise Systems Engineering, University of Illinois,  Urbana, IL, USA
        {\tt\small beck3@illinois.edu}}}%
\begin{document}
\maketitle

\begin{abstract}
Motivated by a variety of applications in control engineering and information sciences, we study network resource allocation problems where the goal is to optimally allocate a fixed amount of resource over a network of nodes. In these problems, due to the large scale of the network with complicated inter-connections between nodes, any solution must be implemented parallel between nodes and based only on local data resulting in the need of distributed algorithms. In this paper, we propose a novel distributed Lagrangian methods, which only requires local computation and communication. Our focus is to understand the performance of this algorithm on the underlying network topology. Specifically, we obtain an upper bound on the rate of convergence of the algorithm as a function of the size and the topology of the underlying network. Finally, to illustrate the effectiveness of the proposed method we apply it to solve the important economic dispatch problems in power systems on the benchmark IEEE 14 and 118 bus systems.      
\end{abstract}

\section{Introduction}\label{sec:Introduction}
We study network resource allocation problems, where the goal is to optimally allocate a fixed portion of resources over a network of nodes. Each node suffers a cost as a function of the portion of resources allocated to it. The goal of this problem is to seek an optimal allocation such that the total cost incurred over the network is minimized, while satisfying each node's local constraints. This problem is sometimes described in terms of utility functions, with the goal being to maximize the total utility. In this problem, due to the large scale and complex interconnection structure of the network, central coordination is undesirable or even impossible, resulting in the necessity for	 distributed algorithms. Our focus, therefore, is to study distributed algorithms which can be implemented in parallel and require only local data.  

Resource allocation is a fundamental and important problem that arises in a variety of application domains within engineering. One standard example is the problem of congestion control where the global objective is to route and schedule information in an internet network such that a fair resource allocation between users is achieved \cite{Srikant2004}. Another example is coverage control in wireless sensor networks, where the goal is to optimally allocate a large number of sensors in an unknown environment such that the coverage area is maximized \cite{Cortes2004, SharmaSB2012}. A simplified version of the important economic dispatch problem in power systems can also be viewed as a resource allocation problem, wherein geographically distributed generators of electricity must coordinate to meet fixed demand while maintaining the stability of the systems \cite{StevenLow2014,Dorfler2015}.

Given the large application domains of network resource allocation problems, distributed algorithms for such problems have received a surge of interest in recent years. The first algorithm which could be implemented in a distributed way was the ``center-free'' protocol introduced in \cite{HoServiSuri80} for a relaxation of this problem, in which the authors consider the case where local constraints are excluded. The term ``center-free'' was originally meant to refer to the absence of any central coordination. In this relaxed problem, the necessary and suffcient optimality conditions are the consensus of the nodes' incremental costs and the feasibility constraint on the total amount of resource available \cite{XiaoBoyd06}. The work in \cite{HoServiSuri80} has been followed by work in \cite{XiaoBoyd06, LakshmananDeFarias08,  Necoara13, Doan17}, with the main focus on analyzing the performance of the algorithm and its variants for the relaxed problem. 

Although the results of the relaxed problem are well-developed, they are impractical since the local constraints of the nodes are critical in real applications. For example, in economic dispatch problems these local constraints, which represent the limited capacity of generators, are inevitable. Recently, there has been a number of studies on resource allocation problems which utilize the necessary and sufficient conditions of the relaxed problem. In particular, the authors in \cite{DominguezGarcia2012, Yang2013, Binetti2014, Kar2012, Xing2015} focus on studying economic dispatch problems where the objective functions are quadratic. Their methods are mainly based on utilizing distributed consensus algorithms to achieve an agreement on the incremental costs between generators while using a proper projection for the generators' local constraints. The authors in  \cite{CortesChekuri2015, LakshmananDeFarias08} relax the assumption on quadratic cost functions to allow for any convex cost functions with Lipschitz continuous gradients. These methods are based on considering a relaxed problem that uses appropriate penalty functions for the nodes' local constraints. While the authors in \cite{LakshmananDeFarias08} only show a convergence to a near-optimal solution, the authors in \cite{CortesChekuri2015} show that the solutions of the relaxed problem coincide with resource allocation problems with an appropriate choice of the penalty level.

In this paper, we propose a fully distributed method for resource allocation problems with general convex objective functions. We focus on analyzing the performance of the method to understand how fast the optimal value may be obtained. In particular, we provide an upper bound for the convergence rate of the method as a function of the size and the topology of underlying networks. This fundamental question, which is necessary to analyze distributed systems, has not been considered explicitly in the literature. Moreover, unlike the work previously cited  where the authors consider either a relaxation of the problem or specific quadratic objective functions, our minimal assumption on the convexity of the objective functions is general enough to cover a larger class of problems.   

\textbf{Main Contributions} The main contributions of the paper are the following. We first propose a novel distributed Lagrangian method for resource allocation problems on undirected and connected networks. Second, we demonstrate that under fairly standard assumptions on the convexity of objective functions and network connectivity, the distributed Lagrangian method achieves a convergence rate $\mathcal{O}(n\ln(k)/(1-\sigma_2)\sqrt{k})$ to the optimal value, where $\sigma_2$ is a parameter representing the spectral properties of network connectivity of the nodes, $n$ is the number of nodes, and $k$ is the number of iterations. Unlike the work in \cite{CortesChekuri2015} and \cite{ LakshmananDeFarias08}, we assume neither any relaxation nor Lipschitz continuity of the gradient of objective functions in our problem, meaning that our method is general enough to cover a larger class of problems. As one illustration of the effectiveness of the proposed method for control and optimization applications, we apply it to solve the important economic dispatch problem in power systems, and specifically consider the benchmark IEEE-14 and IEEE-118 bus test systems. Simulations show that our method outperforms the method proposed in \cite{Kar2012}.

The remainder of this paper is organized as follows. We give a formal statement of the problem in Section \ref{sec:ProbForm}. Our proposed method is described in Section \ref{sec:Algorithm} and its convergence analysis is given in Section \ref{sec:Analysis}. In Section \ref{sec:Simulation}, we discuss the results of applying our  method to specific economic dispatch problems via simulations. We conclude the paper with some potential future directions in Section \ref{sec:Conclusion}.

\begin{notation}
We use boldface to distinguish between vectors $\mathbf{x}$ in $\mathbb{R}^n$ and scalars $x$ in $\mathbb{R}$. Given a vector $\mathbf{x}$, we write $\mathbf{x}=(x_1,x_2,\ldots,x_n)$ and denote by $\|\mathbf{x}\|_2$ its Euclidean norm. Finally, let $\mathbf{1}\in\mathbb{R}^n$ be the vector whose entries are $1$.
\end{notation}

\section{Problem Formulation}\label{sec:ProbForm}
In this paper we consider resource allocation problems over a network of $n$ nodes. The goal of this problem is to seek an optimal allocation for a prespecified portion or quantity of a given resource such that the total cost incurred over the network is minimized while satisfying the nodes' local constraints. In general, this problem can be formulated as the following optimization problem:
\begin{align}\label{prob:OPI}
\text{{\sf P} \ :}\; \left\{ \begin{array}{ll}
\min_{\mathbf{x}\in\mathbb{R}^n} \sum_{i=1}^n f_i(x_i)\vspace{0.2cm}\\
\text{s.t. }\;\sum_{i=1}^{n}x_i =  b,\\
\quad\quad x_i\in\mathcal{X}_i,\quad\forall i = 1,\ldots,n.
\end{array}\right.
\end{align}
where $f_i:\mathbb{R}\rightarrow\mathbb{R}$ is a cost function known only by node $i$ and the constant $b$ is the total portion of resources shared by the nodes. Moreover, the sets $\mathcal{X}_i$ are assumed to be convex and compact.  We note that the formulation of {\sf P} is general enough to cover the models studied in \cite{Doan17, CortesChekuri2015, LakshmananDeFarias08, Kar2012, DominguezGarcia2012}.

In the sequel, will use $\mathcal{X}$ to denote the Cartesian product of $\mathcal{X}_i$, i.e., $~\mathcal{X} = \mathcal{X}_1\times\mathcal{X}_2\times\ldots\times\mathcal{X}_n$. The feasible set $\mathcal{S}$ of {\sf P} is given as  $\mathcal{S}=\{\mathbf{x}\in\mathcal{X} | \sum_{i=1}^n x_i = b\}$ and it is assumed to be nonempty. For short, we denote $f:\mathbb{R}^n\rightarrow\mathbb{R}$ as the sum of the functions $f_i$, i.e.,
\begin{align}
f(\mathbf{x}) = \sum_{i=1}^n f_i(x_i).\label{prob:objFunction}
\end{align}
We will make the following assumptions on the objective functions $f_i$ throughout the paper:
\begin{assumption}\label{assump:ConvexAssump}
For each $i=1,\ldots,n$ the function $f_i:\mathbb{R}\rightarrow\mathbb{R}$ is proper, closed, and convex.
\end{assumption}

\begin{assumption}[Slater's condition \cite{Bertsekas2003}]\label{assump:Slater}
There exists a point $\tilde{\xbf}$ that belongs to the relative interior of $\Xcal$ and satisfies $\sum_{i=1}^n \tilde{x}_i = b$.
\end{assumption}

We note that since the constraint set in {\sf P} is compact and the objective function $f$ is continuous, there exists a vector $\mathbf{x}^*=(x_1^*,x_2^*,\ldots,x_n^*)\in\mathcal{S}$ which achieves the minimum of this problem. However, this solution may not be unique. We denote the set of solutions of {\sf P} as $\mathcal{S}^*$. Moreover, under the Slater's condtion strong duality holds and the set of dual optimal solutions is nonempty.

Our focus is on designing a distributed method to solve {\sf P} in a network, which means that each node is only allowed to interact with nodes, referred to as their local neighbors, that are connected to it in a graph. Specifically, we assume we are given a graph $\mathcal{G} = (\mathcal{V},\mathcal{E})$ with $\mathcal{V} = \{1,\ldots,n\}$; nodes $i$ and $j$ can exchange messages if and only if $(i,j) \in \mathcal{E}$. We denote by $\mathcal{N}_i$ the set of neighbors of node $i$. We make the following fairly standard assumption regarding connectivity of graph $G$.
\begin{assumption}\label{assump:Connectivity}
The graph $G$ is undirected and connected. 
\end{assumption}

\section{Main Algorithm}\label{sec:Algorithm}
In this section, we propose a novel distributed method, namely, a distributed Lagrangian method, for {\sf P}. Our method can be presented as a distributed version of well-known classical Lagrangian methods \cite{Bertsekas1999}. We start this section by giving a brief motivation for our method.

Without loss of generality, we assume that each node $i$ knows the constant $b_i$ such that $\sum_{i=1}^n b_i= b$. Here $b_i$ can be interpreted as the initial resource allocation at node $i$. One specific choice is to initially distribute $b$ equally to all nodes, i.e., $b_i=b/n$ for all $i = 1,\ldots,n$.  We note that the design of our algorithm as well as our analysis given later does not depend on the choice of these $b_i$ since they are only used for notational convenience. 

We now explain the mechanics of our approach. Consider the following Lagrangian function ${\Lcal:\Xcal\times\Rset\rightarrow\Rset}$ of ${\sf P}$ 
\begin{align}
\Lcal(\xbf,\lambda) := \sum_{i=1}^n f_i(x_i) + \lambda\left(\sum_{i=1}^n x_i - b_i\right),\label{alg:Lagrangian}
\end{align} 
where $\lambda\in\Rset$ is the Lagrangian multiplier associated with the coupling equality constraint in \eqref{prob:OPI}. The dual  function $d:\Rset\rightarrow\Rset$ of problem ${\Psf}$ for some $\lambda$ value is then defined as
\begin{align}
d(\lambda) &:= \underset{\xbf\in\Rset^n}{\text{min}}\left( \sum_{i=1}^n f_i(x_i) + \lambda \left(\sum_{i=1}^n (x_i-b_i)\right) \right)\nonumber\\
&= \sum_{i=1}^n \{-f_i^*(-\lambda)-\lambda b_i\},\nonumber 
\end{align} 
where $f_i^*$ is the Fenchel conjugate of $f_i$, defined by 
\begin{align*}
f_i^*(u) = \underset{x\in\Xcal_i}{\sup}\{ux - f_i(x)\}.
\end{align*}
The dual problem of ${\Psf}$, denoted by ${\sf DP}$, is given by
\begin{align}
{\sf DP}\ : \ \underset{\lambda\in\Rset}{\text{max}} \sum_{i=1}^n \{-f_i^*(-\lambda)-\lambda b_i\},\nonumber	
\end{align} 
which is then equivalent to solving 
\begin{align}
\underset{\lambda\in\Rset}{\text{min}} \sum_{i=1}^n \underbrace{f_i^*(-\lambda)\ + \ \lambda b_i}_{=q_i(\lambda)}, \label{alg:mindual}
\end{align}
where $q_i \ : \ \Rset\rightarrow\Rset$ is a convex function since $f_i^*$ is a convex function \cite{Bertsekas1999}. 

A common approach to solve problem \eqref{prob:OPI} is  Lagrangian methods \cite{Bertsekas1999}, which require a central coordinator to update and distribute the multiplier $\lambda$ to the nodes. The key idea of our approach is to eliminate this requirement by utilizing the distributed subgradient method presented in \cite{Nedic2009} to compute the solution of \eqref{alg:mindual}. In particular, we have that each node $i$ stores a local copy $\lambda_i$ of the Lagrange multiplier $\lambda$, and then iteratively updates $\lambda_i$ upon communicating with its neighbors. The update of $\lambda_i$ is comprised of two steps, namely, a consensus step and a so-called ``local gradient descent" step. These two steps, coupled with the primal update in the Lagrangian approach, results in our distributed Lagrangian algorithm, presented in Algorithm \ref{alg:DLM}. 

The updates in Algorithm \ref{alg:DLM} have a simple implementation: first, at time $k\geq 0$, each node $i$  in step \eqref{DLM:viUpdate} broadcasts the current value of $\lambda_i(k)$ to its neighbors. Node $i$ computes $v_i$, as given by the weighted average of the current local values of its neighbors and its own value. The updates of $x_i(k+1)$ and $\lambda_i(k+1)$ then do not require any additional communications among the nodes at step $k$. In particular, the update of the primal variables $x_i$ is executed in step \eqref{DLM:xiUpdate}. Finally, in step \eqref{DLM:LambdaiUpdate} $\lambda_i$ is updated by moving a distance of $\alpha(k)$ along a subgradient of $q_i$ at $v_i(k+1)$, given by
\begin{equation}
b_i - x_i(k+1) \in \partial q_i(v_i(k+1)),\label{alg_DLM:isubgradient}
\end{equation}  
where $\partial q_i(v_i(k+1))$ is the sub-differential set of $q_i$ at ${v_i(k+1)}$. We note that our method maintains the feasibility of the nodes' local constraints at every iteration, i.e., $x_i(k)\in\mathcal{X}_i$ for all $k\geq 0$. Furthermore, the nodes do not need to store the variables $v_i$ in step \eqref{DLM:viUpdate} since these are used only for notational convenience.


Let $A$ be the communication matrix whose the $(i,j)-$th entries are $a_{ij}$. Moreover, for short we denote the function $q:\mathbb{R}^n\rightarrow\mathbb{R}$ as $$q(\bld) = \sum_{i=1}^n q_i(\lambda_i).$$ The updates \eqref{DLM:viUpdate}--\eqref{DLM:LambdaiUpdate} then can be written compactly in vector form as,
\begin{align}
&\mathbf{v}(k+1) = A\bld(k)\nonumber\\
&\mathbf{x}(k+1)\in \arg\min_{\mathbf{x}\in\mathcal{X}}\sum_{i\in \mathcal{V}}f_i(x_i) + v_i(k+1)(x_i-b_i) \nonumber\\
&\bld(k+1) = \mathbf{v}(k+1) - \alpha(k)\nabla q(\mathbf{v}(k+1)),\nonumber
\end{align}
where $\nabla q(\cdot)$, with some abuse of notation, is the subgradient of the dual function $q$, i.e., $$\nabla q(\bld) = [\partial q_1(\lambda_1),\ldots,\partial q_n(\lambda_n)]^T.$$

Throughout the remainder of this paper, we assume that $A$ satisfies the following assumption.
\begin{assumption}\label{assump:doublystochastic}
$A$ is a doubly stochastic matrix, i.e., $\sum_{i=1}^n a_{ij} = \sum_{j=1}^n a_{ij} = 1$ with $a_{ii}>0$. Moreover, the weights $a_{ij} > 0$ if and only if $(i,j)\in \mathcal{E}$ otherwise $a_{ij} = 0$.
\end{assumption}

In the sequel, let $\sigma_2$ be the second largest singular value of $A$. By the Courant-Fisher theorem \cite{HJbook}, since $A$ is doubly stochastic and the graph is connected, we have $\sigma_2\in (0,1)$.

\begin{algorithm}
\caption{A Distributed Lagrangian Method (DLM)}\label{alg:DLM}
1. \textbf{Initialize}: Each node $i$ initializes its variables as $x_i(0) \in \Rset$,  $\lambda_i(0)\in\mathbb{R}$.\\
2. \textbf{Iteration}: For $k\geq 0$ each node $i$ implements the following steps
\smallskip
\vspace{-2mm}
\begin{align}
&v_i(k+1) = \sum_{j\in \mathcal{N}_i}a_{ij}\lambda_j(k)\label{DLM:viUpdate}\\
&x_i(k+1)\!\in\! \arg\!\!\min_{x_i\in\mathcal{X}_i}\!f_i(x_i)\! +\! v_i(k+1)(x_i-b_i)\label{DLM:xiUpdate}\\
&\lambda_i(k+1) = v_i(k+1) - \alpha(k)(x_i(k+1)-b_i)\label{DLM:LambdaiUpdate}
\end{align}
\end{algorithm}

\section{Convergence Analysis}\label{sec:Analysis}
The goal of this section is to provide an analysis for the convergence of the distributed Lagrangian method proposed in the previous section. Specifically, we will show that the method achieves an asymptotic convergence to the optimal value of {\sf P}. In addition, to understand how fast the algorithm obtains the optimal objective value we utilize the standard technique from centralized subgradient methods. Specifically, the method achieves an asymptotic convergence to the optimal value at the rate $\mathcal{O}(n\ln(k)/(1-\sigma_2)\sqrt{k})$. Here the parameter $\sigma_2$ represents the spectral properties of network connectivity. We note that such an explicit upper bound on the convergence rate for distributed resource allocation is not available in the literature. More detail will be given shortly. 

We start our analysis by introducing more notation. Given an optimal solution $\mathbf{x}^*\in\mathcal{S}^*$ of {\sf P},  let $\lambda^*$ be a dual optimal of {\sf DP}, i.e., $(\mathbf{x}^*,\lambda^*)$ is a saddle point of $\mathcal{L}(\xbf,\lambda)$ 
$$\Lcal(\xbf^*,\lambda)\leq\Lcal(\xbf^*,\lambda^*)\leq\Lcal(\xbf,\lambda^*)\;\ {\forall \xbf\in\mathbb{R}^n,\lambda\in\mathbb{R}}.$$

Moreover, let $\bld^*$ be a vector such that ${\bld^* = (\lambda^*,\ldots,\lambda^*)^T\in\mathbb{R}^n}$. Denote by $\mathcal{L}_i:\mathcal{X}_i\times\mathbb{R}\rightarrow\mathbb{R}$ the local Lagrangian function at node $i$
\begin{align*}
\mathcal{L}_i(x_i,v_i) = f_i(x_i) + v_i(x_i-b_i),
\end{align*}
and with some abuse of notation, the global Lagrangian function $\mathcal{L}:\mathcal{X}\times\mathbb{R}^n\rightarrow\mathbb{R}$ is defined as
\begin{align}
\mathcal{L}(\mathbf{x},\mathbf{v}) = \sum_{i\in\mathcal{V}} \mathcal{L}_i(x_i,v_i).\label{RAPCR:Lagragian}
\end{align}
Our first result is the Lipschitz continuity of the functions $q_i$ given in the following proposition.

\begin{proposition}\label{proposition:subgbound}
Let the sequences $\{x_i(k)\}$ and $\{\lambda_i(k)\}$ for all $i\in\mathcal{V}$ be generated by the DLM algorithm. Then there exists a positive constant $C$ such that for all $k\geq 0$,
\begin{align}
|\partial q_i(v_i(k+1))|\leq C,\quad \text{ for all }  i\in \mathcal{V}.\label{prop:idualSB}
\end{align}
\end{proposition}
\vspace{0.3cm}
\begin{proof}
Since $\mathcal{X}_i$ is a compact set and $x_i(k)\in\mathcal{X}_i$ $\forall k\geq 0$, by \eqref{alg_DLM:isubgradient} there exists a positive constant $C_i$ such that $|\partial q_i(v_i(k))|\leq C_i$ $\forall k$. Letting $C=\max_{i} C_i $ gives \eqref{prop:idualSB}.
\end{proof}
\vspace{0.2cm}
Recall that the updates \eqref{DLM:viUpdate} and \eqref{DLM:LambdaiUpdate} are the distributed subgradient methods, studied in \cite{Nedic2009} for problem \eqref{alg:mindual}. Moreover, since the subgradient of $q_i$ is bounded, we can now apply the result in \cite[Proposition 4]{Nedic2009} to show the convergence of $\lambda_i$ to $\lambda^*$, a dual optimal. Due to the limited space, we skip the proof here and refer the readers to \cite{Nedic2009}.  

\begin{lemma}[\cite{Nedic2009}]\label{CLlem:DualConvergence}
Let Assumptions \ref{assump:ConvexAssump}--\ref{assump:doublystochastic} hold. Let the sequences $\{x_i(k)\}$ and $\{\lambda_i(k)\}$ for all $ i\in \mathcal{V}$ be generated by the DLM algorithm. Assume that the step size $\alpha(k)$ is positive, nonincreasing with $\alpha(0) = 1$, and satisfies the following conditions,
\begin{align}
\sum_{k=0}^\infty \alpha(k) = \infty,\quad \sum_{k=0}^\infty \alpha^2(k) < \infty.\label{CLlem:stepsize}
\end{align}
Then, given a dual optimal $\lambda^*$ of \eqref{alg:mindual}, we have
\begin{align}
\lim_{k\rightarrow\infty}\lambda_i(k) = \lambda^*,\quad \text{for all } i\in\mathcal{V}.\label{CLlem:dualconv}
\end{align}
\end{lemma}
\vspace{0.2cm}
A specific choice of the squence of stepsize is $\alpha(k) = 1/k$, which obviously satisfies \eqref{CLlem:stepsize}. We are now ready to state our first main result, that is the distributed Lagrangian method achieves an asymptotic convergence to the optimal value of problem {\Psf}\footnote{There is an error in Theorem 1 in the version appeared on IEEE Conference on Control and Technology Applications 2017 \cite{DoanB2016}, which is corrected by Theorem \ref{CLthm:primalconvergence} given here}. 

\begin{theorem}\label{CLthm:primalconvergence}
Let all assumptions given in Lemma \ref{CLlem:DualConvergence} hold. Let $\{x_i(k)\}$ and $\{\lambda_i(k)\}$ for all $i\in\mathcal{V}$ be generated by Algorithm \ref{alg:DLM}. Then, for all integers $k\geq 0$ we have,
\vspace{-0.2cm}
\begin{align}
\lim_{k\rightarrow\infty}\mathcal{L}(\mathbf{x}(k+1),\boldsymbol{\lambda}(k))=f(\xbf^*).\label{CLthm:asympconv}
\end{align}
\end{theorem}
\vspace{0.3cm}
\begin{proof}
We first have $(\xbf^*,\lambda^*)$ is a saddle point of the Lagrangian \eqref{alg:Lagrangian} where $\lambda^*$ satisfies \eqref{CLlem:dualconv}. Recall from \eqref{RAPCR:Lagragian} that
\begin{align}
\mathcal{L}(\mathbf{x},\mathbf{v}) = \sum_{i\in\mathcal{V}}\mathcal{L}_i(x_i,v_i) = \sum_{i\in\mathcal{V}}f_i(x_i) + v_i(x_i-b_i).\nonumber
\end{align}
To show our main result, we will show the following relation,
\begin{align}
&0\leq \sum_{i\in\mathcal{V}}\mathcal{L}_i(x_i^*,v_i(k+1))-\mathcal{L}_i(x_i(k+1),v_i(k+1))\nonumber\\
&\leq C\|\mathbf{v}(k+1)-\bld^*\|_{1}  + \sum_{i\in\mathcal{V}}v_i(k+1)(x_i^*-b_i).\label{CLthm:Eq2}
\end{align}
We note that by Lemma \ref{CLlem:DualConvergence}, we have $\lim_{k\rightarrow\infty}\lambda_i(k)= \lambda^*$ $\forall i\in\mathcal{V}$, implying $\lim_{k\rightarrow\infty}v_i(k+1) = \lambda^*$ since $ v_i(k+1) = \sum_{j\in\mathcal{N}_i}a_{ij}\lambda_j(k)$ and $A$ is a doubly stochastic matrix. In addition, since $x_i^*$ satisfies  $\sum_{i\in\mathcal{V}}x_i^*=\sum_{i\in\mathcal{V}}b_i= b$ we have
\begin{align}
&\lim_{k\rightarrow\infty} \|\mathbf{v}(k+1)-\bld^*\|_{1} = 0,\nonumber\\
&\lim_{k\rightarrow\infty}\sum_{i\in\mathcal{V}}v_i(k+1)(x_i^*-b_i) = 0.\nonumber
\end{align}
Thus, we obtain
\begin{align}
0&\leq\lim_{k\rightarrow\infty}\sum_{i\in\mathcal{V}}\mathcal{L}_i(x_i^*,v_i(k+1))-\mathcal{L}_i(x_i(k+1),v_i(k+1))\nonumber\\
&=\lim_{k\rightarrow\infty}\sum_{i\in\mathcal{V}}\mathcal{L}_i(x_i^*,\lambda^*)-\mathcal{L}_i(x_i(k+1),v_i(k+1)) = 0,\nonumber
\end{align}
which implies that
\begin{align}
&0\leq \mathcal{L}(\mathbf{x}^*,\bld^*) - \limsup_{k\rightarrow\infty}\mathcal{L}(\mathbf{x}(k+1),\bld(k))\nonumber\\
&= f(\mathbf{x^*}) - \limsup_{k\rightarrow\infty}\mathcal{L}(\mathbf{x}(k+1),\bld(k))= 0,\label{CLthm:Eq3}\displaybreak[0]
\end{align}
where we use the fact that $\mathcal{L}(\mathbf{x}^*,\bld^*) = f(\xbf^*)$ and $\lim_{k\rightarrow\infty}\bld(k) = \lim_{k\rightarrow\infty}\mathbf{v}(k) = \bld^*$.

We are now proceeding to show \eqref{CLthm:Eq2}. Since $x_i(k+1)$ satisfies \eqref{DLM:xiUpdate} and by the definition of $\mathcal{L}_i$ in \eqref{RAPCR:Lagragian} we have for $\forall i\in\mathcal{V}$ and $k\geq 0$,
\begin{align}
0 \leq \mathcal{L}_i(x_i^*,v_i(k+1)) - \mathcal{L}_i(x_i(k+1),v_i(k+1)),\nonumber
\end{align}
which when summing for all $i\in\mathcal{V}$ implies that
\begin{align}
0&\leq\sum_{i\in\mathcal{V}}\mathcal{L}_i(x_i^*,v_i(k+1))-\mathcal{L}_i(x_i(k+1),v_i(k+1))\nonumber\\
&= \sum_{i\in\mathcal{V}} f_i(x_i^*) + v_i(k+1)(x_i^*-b_i)\nonumber\\
&-\sum_{i\in\mathcal{V}} f_i(x_i(k+1)) + v_i(k+1)(x_i(k+1)-b_i).\label{CLthm:Eq6}
\end{align}
By Assumption \ref{assump:ConvexAssump} the zero duality gap holds, i.e.,
\begin{align}
\sum_{i\in\mathcal{V}} f_i(x_i^*) = \sum_{i\in\mathcal{V}}d_i(\lambda^*) \stackrel{\eqref{alg:mindual}}{=} -\sum_{i\in\mathcal{V}}q_i(\lambda^*).\nonumber
\end{align}
Moreover by \eqref{alg:mindual} and \eqref{DLM:xiUpdate} we have
\begin{align}
q_i(v_i(k+1))\!\!=\!- f_i(x_i(k+1))\!\!- v_i(k+1)(x_i(k+1)-b_i).\nonumber
\end{align}
Substitute the previous two preceding relations into \eqref{CLthm:Eq6} we obtain
\begin{align}
0&\leq\sum_{i\in\mathcal{V}}\mathcal{L}_i(x_i^*,v_i(k+1))-\mathcal{L}_i(x_i(k+1),v_i(k+1))\nonumber\\
&= \sum_{i\in\mathcal{V}} q_i(v_i(k+1)) - q_i(\lambda^*) + v_i(k+1)(x_i^*-b_i)\nonumber\\
&\leq \sum_{i\in\mathcal{V}} C|\lambda^*-v_i(k+1)|  + v_i(k+1)(x_i^*-b_i),\nonumber
\end{align}
where the last inequality is due to \eqref{prop:idualSB}. This concludes our proof.
\end{proof}
To show the convergence rate of the distributed Lagrangian method, we need the following properties of \eqref{DLM:LambdaiUpdate} studied in \cite{Nedic2015}, for the case of undirected connected graph.

\begin{lemma}[Lemma 3 \cite{Nedic2015}]\label{lem:pertbconsensus}
Let Assumptions \ref{assump:Connectivity}--\ref{assump:doublystochastic} hold. Let the sequence $\{\lambda_i(k)\}$ for all $i\in\mathcal{V}$ be generated by Algorithm \ref{alg:DLM}. Assume that the step size $\alpha(k)$ is nonincreasing with $\alpha(0) = 1$. Then the following statements hold,
\begin{enumerate}
\item  For all $i=1,\ldots,n,$ and $k\geq 0$
\begin{align}
&|\lambda_i(k) - \bar{\lambda}(k)|\nonumber\\
&\leq \sigma_2^k\|\bld(0)\|_1+ \sqrt{n}C\sum_{t=0}^{k-1} \alpha(t)\sigma_2^{k-1-t},\label{pertblem:averbound}
\end{align}
where $\sigma_2\in (0,1)$ is the second largest singular value of $A$.
\item If in addition $\alpha(k) = 1/\sqrt{k}$ for $k\geq 1$ and $\alpha(0) = 1$, then for all $i=1,\ldots,n,$ and an integer $K\geq 1$
\begin{align}
&\sum_{k=0}^K\!\! \alpha(k)|\lambda_i(k) - \bar{\lambda}(k)|\nonumber\\
&\leq \frac{
\|\bld(0)\|_{1}}{1-\sigma_2}+ \frac{\sqrt{n}C(2+\ln(K))}{1-\sigma_2}\cdot\label{pertblem:averrate}
\end{align}
\end{enumerate}
\end{lemma}
\vspace{0.3cm}
We now ready to show that the distributed Lagrangian method converges to the optimal value at a rate $\mathcal{O}(n\ln(k)/(1-\sigma_2)\sqrt{k})$ on the dual function value estimated at the time-weighted average of any dual estimate. The following Theorem   adopted from \cite{Nedic2015}, is to show this result.
\begin{theorem}[\cite{Nedic2015}]\label{thm:convrate}
Let Assumptions \ref{assump:ConvexAssump}--\ref{assump:doublystochastic} hold. Let $\{x_i(k)\}$ and $\{\lambda_i(k)\}$ for all $i\in\mathcal{V}$ be generated by Algorithm \ref{alg:DLM}. Let $\alpha(k) = 1/\sqrt{k}$ for $k\geq 1$ and $\alpha(0) = 1$. Then we have for all $i=1,\ldots,n$, and  $K\geq 1$
\begin{align}
&q\left(\!\!\!\begin{array}{l}
\frac{\sum_{k=0}^K\alpha(k)\lambda_i(k)}{\sum_{k=0}^K\alpha(k)}
\end{array}\!\! \mathbf{1}\right) - q(\bld^*)\nonumber\\
&\leq\! \frac{\|\bld(0)\!-\!\bld^*\|_2^2}{4\sqrt{K}}+ \frac{4\sqrt{n}C\|\bld(0)\|_{1}+ 5nC^2(2+\ln(K))}{4(1-\sigma_2)\sqrt{K}}\cdot\label{CLcorollary:functionrate}
\end{align}
\end{theorem}

\begin{proof}
Let $\lambda^*$ be a dual optimal of \eqref{alg:mindual}. By \eqref{DLM:viUpdate} and \eqref{DLM:LambdaiUpdate} we have for all $\lambda\in\mathbb{R}$,
\begin{align}
&\sum_{i\in\mathcal{V}}(\lambda_i(k+1) - \lambda^*)^2\nonumber\\
&= \sum_{i\in\mathcal{V}} (v_i(k+1) - \alpha(k)\partial q_i(v_i(k+1))-\lambda^*)^2\nonumber\\
&=  \sum_{i\in\mathcal{V}}(v_i(k+1)-\lambda^*)^2 + \sum_{i\in\mathcal{V}}\alpha^2(k)[\partial q_i(v_i(k+1))]^2 \nonumber\\
&\quad-2\alpha(k)\sum_{i\in\mathcal{V}}\partial q_i(v_i(k+1))(v_i(k+1) - \lambda^*)\nonumber\\
&\leq \|\mathbf{v}(k+1)-\bld^*\|_2^2 +nC^2\alpha^2(k)\nonumber\\
&\quad+ 2\alpha((k)q(\bld^*)-q(\mathbf{v}(k+1)),
 \label{URlem:Eq1}
\end{align}
where the last inequality is due to \eqref{prop:idualSB}, the convexity of each $q_i$. 
By \eqref{prop:idualSB} and for some $i\in\mathcal{V}$ the preceding relation implies 
\begin{align}
&\|\bld(k+1) - \bld^*\|_2^2\nonumber\\
&\leq \|\mathbf{v}(k+1)-\bld^*\|_2^2 +nC^2\alpha^2(k)\nonumber\\
&\;+2\alpha(k)(q(\bld^*)-q(\lambda_i(k)\mathbf{1}))\nonumber\\
&\;+ 2\alpha(k)\Big(q(\lambda_i(k)\mathbf{1})-q(\bar{\bld}(k))+q(\bar{\bld}(k))-q(\mathbf{v}(k+1))\Big)\nonumber\\
&\leq  \|\mathbf{v}(k+1)-\bld^*\|_2^2+ nC^2\alpha^2(k)\nonumber\\
&\;+2\alpha(k)(q(\bld^*)-q(\lambda_i(k)\mathbf{1}))\nonumber\\
&\;+ 2C\alpha(k)(\|\lambda_i(k)\mathbf{1}-\bar{\bld}(k)\|_1+\|\bar{\bld}(k)-\mathbf{v}(k+1)\|_1)\nonumber\\
&\leq  \|\bld(k)-\bld^*\|_2^2+ nC^2\alpha^2(k)+2\alpha(k)(q(\bld^*)-q(\lambda_i(k)\mathbf{1}))\nonumber\\
&\;+ 2C\alpha(k)(\sqrt{n}|\lambda_i(k)-\bar{\lambda}(k)|+\|\bar{\bld}(k)-\bld(k)\|_1),\label{URlem:Eq2}
\end{align}
where the last inequality is due to Jensen's inequality and the double stochasticity of $A$, i.e.,
\begin{align}
&\|\mathbf{v}(k+1)-\bar{\bld}(k)\|_1 \leq \|\bld(k)-\bar{\bld}(k)\|_1.
\end{align}
Sum up both sides of \eqref{URlem:Eq2} from $k=0,\ldots,K$ for some $K\geq0$, and use \eqref{pertblem:averrate} we have
\begin{align}
&\|\bld(K+1) - \bld^*\|_2^2\nonumber\\
&\leq  \|\bld(0)-\bld^*\|_2^2+ nC^2\sum_{k=0}^K\alpha^2(k)\nonumber\\
&\quad+2\sum_{k=0}^K\alpha(k)(q(\bld^*)-q(\lambda_i(k)\mathbf{1}))\nonumber\\
&\quad+ 2C\sum_{k=0}^K\alpha(k)(\sqrt{n}|\lambda_i(k)-\bar{\lambda}(k)|+\|\bar{\bld}(k)-\bld(k)\|_1)\nonumber\\
&\leq  \|\bld(0)-\bld^*\|_2^2+ nC^2\sum_{k=0}^K\alpha^2(k)\nonumber\\
&\quad+2\sum_{k=0}^K\alpha(k)(q(\bld^*)-q(\lambda_i(k)\mathbf{1}))\nonumber\\
&\quad+ \frac{4\sqrt{n}C\|\bld(0)\|_{1}}{1-\sigma_2}+ \frac{4nC^2(2	+\ln(K))}{1-\sigma_2}.\nonumber
\end{align}
Dividing both sides of the previous equation by $2\sum_{k=0}^K\alpha(k)$ and rearranging the terms we obtain 
\begin{align}
&\frac{\sum_{k=0}^K\alpha(k)q(\lambda_i(k)\mathbf{1})}{\sum_{k=0}^K\alpha(k)}-q(\bld^*)\nonumber\\
&\leq  \frac{\|\bld(0)-\bld^*\|_2^2+ nC^2\sum_{k=0}^K\alpha^2(k)}{2\sum_{k=0}^K\alpha(k)}\nonumber\\
&\quad+ \frac{2\sqrt{n}C\|\bld(0)\|_{1}+2nC^2(2+\ln(K))}{(1-\sigma_2)\sum_{k=0}^K\alpha(k)}.\label{URlem:Eq3}
\end{align}
Since $\alpha(k) = 1/\sqrt{k}$ with $\alpha(0) = 1$, we have
\begin{align}
&\sum_{k=0}^K\alpha^2(k) = 2 + \sum_{k=2}^K\frac{1}{k} \leq 2+ \int_{1}^K\frac{dk}{\ln(k)} = 2+\ln(K).\nonumber\\
&\sum_{k=0}^K\alpha(k) = \sum_{k=0}^K\frac{1}{\sqrt{k}} \geq\int_{0}^K\frac{dk}{\sqrt{k}} = 2\sqrt{K}\nonumber
\end{align}
Substituting these two preceding relations into \eqref{URlem:Eq3} and using Jensen's inequality for the left hand side we obtain
\begin{align}
&q\left(\!\!\!\begin{array}{l}
\frac{\sum_{k=0}^K\alpha(k)\lambda_i(k)}{\sum_{k=0}^K\alpha(k)}
\end{array}\!\!\mathbf{1}\!\right)-q(\bld^*)\nonumber\\
&\leq  \frac{\|\bld(0)-\bld^*\|_2^2+ nC^2(2+\ln(K))}{4\sqrt{K}}\nonumber\\
&\quad+ \frac{2\sqrt{n}C\|\bld(0)\|_{1}+2nC^2(2+\ln(K))}{2(1-\sigma_2)\sqrt{K}}\nonumber\\
&\leq  \frac{\|\bld(0)-\bld^*\|_2^2}{4\sqrt{K}}+ \frac{\sqrt{n}C\|\bld(0)\|_{1}}{(1-\sigma_2)\sqrt{K}}\nonumber\\
&\quad+ \frac{5nC^2(2+\ln(K))}{4(1-\sigma_2)\sqrt{K}}\nonumber,
\end{align}
which concludes our proof. 
\end{proof}

\section{Power Application Studies}\label{sec:Simulation}
Our main motivations for the studies in this paper are for use in power systems applications \cite{StevenLow2014,Dorfler2015}. In this section, we show the effectiveness of the Lagrangian distributed method by applying it to the economic dispatch problem, wherein geographically distributed generators of electricity must coordinate to meet fixed demand. We first apply the method to the IEEE-14 bus test system \cite{IEEE14BusWebsite} and compare the performance of our method with that in \cite{Kar2012}. Then to illustrate its performance on a larger scale system, we study economic dispatch on the IEEE-118 bus test system \cite{IEEE118BusWebsite}. Simulation results are given to demonstrate that our method works very well in both systems.

\subsection{Case study 1: Economic dispatch on IEEE-14 bus systems}

\begin{figure}
\centering
\includegraphics[width=3.5in,keepaspectratio]{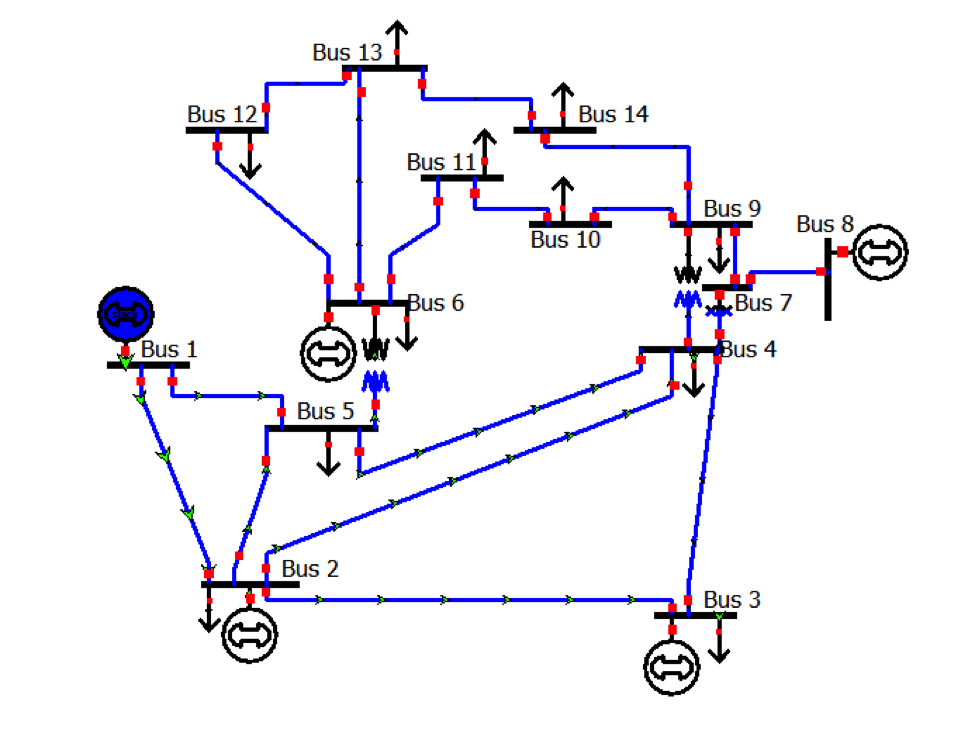}
\caption{IEEE $14$ bus systems.}\label{fig:IEEE14Bus}
\end{figure}

\begin{table}
\centering
\caption{Generator parameters (MU= Monetary units)}
\vspace{-0.2cm}
\begin{tabular}{|c|c|c|c|c|}
\hline
Gen. & Bus & $\gamma_i[MU/MW^2]$ & $\beta_i[MU/MW]$ & $P_i^{\text{max}}[MW]$\\
\hline
 1 & 1 & 0.04 & 2.0 & 80 \\
 2 & 2 & 0.03 & 3.0 & 90 \\
 3 & 3 & 0.035 & 4.0 & 70 \\
 4 & 6 & 0.03 & 4.0 & 70 \\
 5 & 8 & 0.04 & 2.5 & 80 \\
\hline
\end{tabular}
\label{tab:14BusGeneratorPar}
\end{table}

\begin{figure}
\centering
\includegraphics[width=3.5in,keepaspectratio]{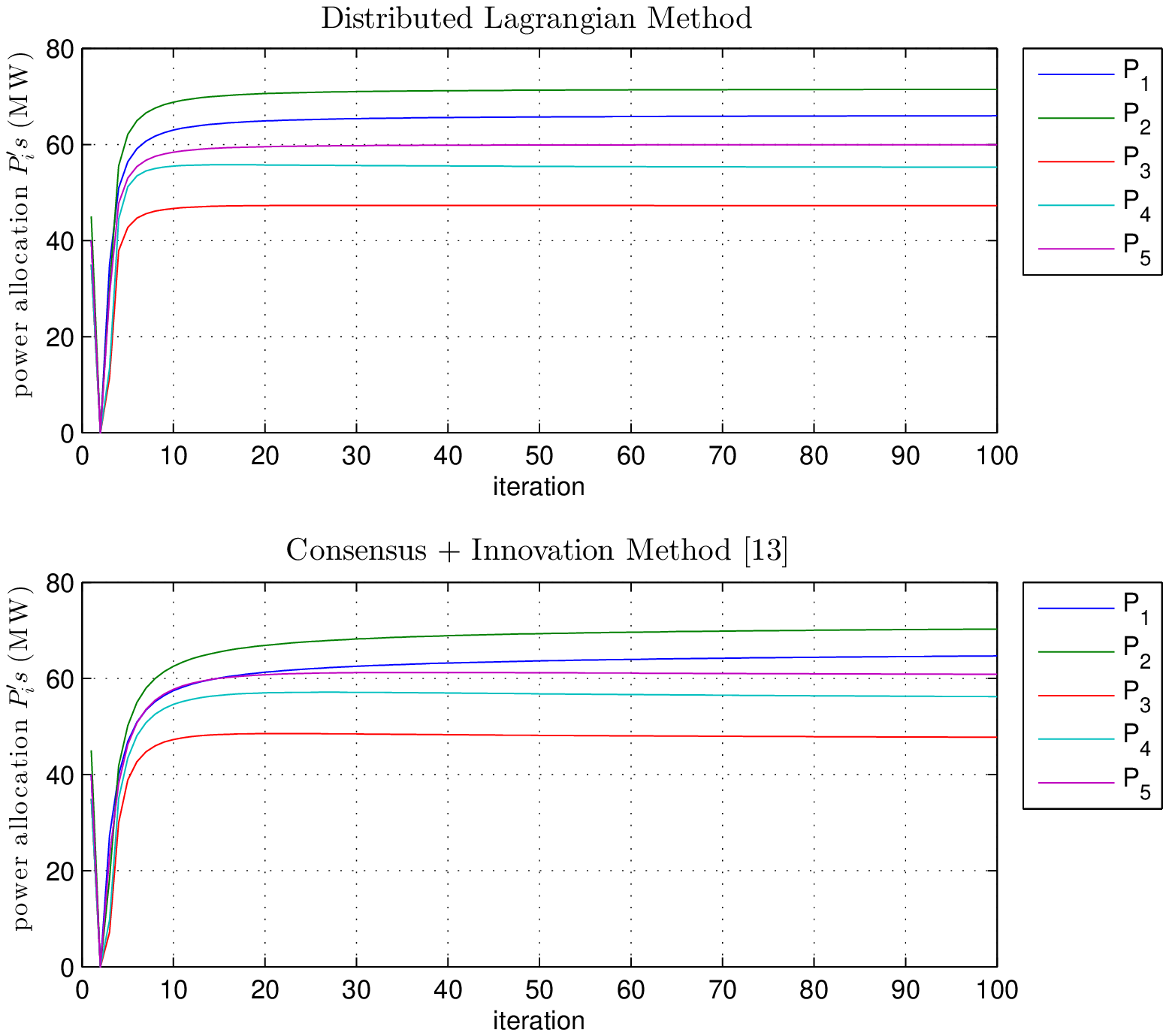}
\includegraphics[width=3.5in,keepaspectratio]{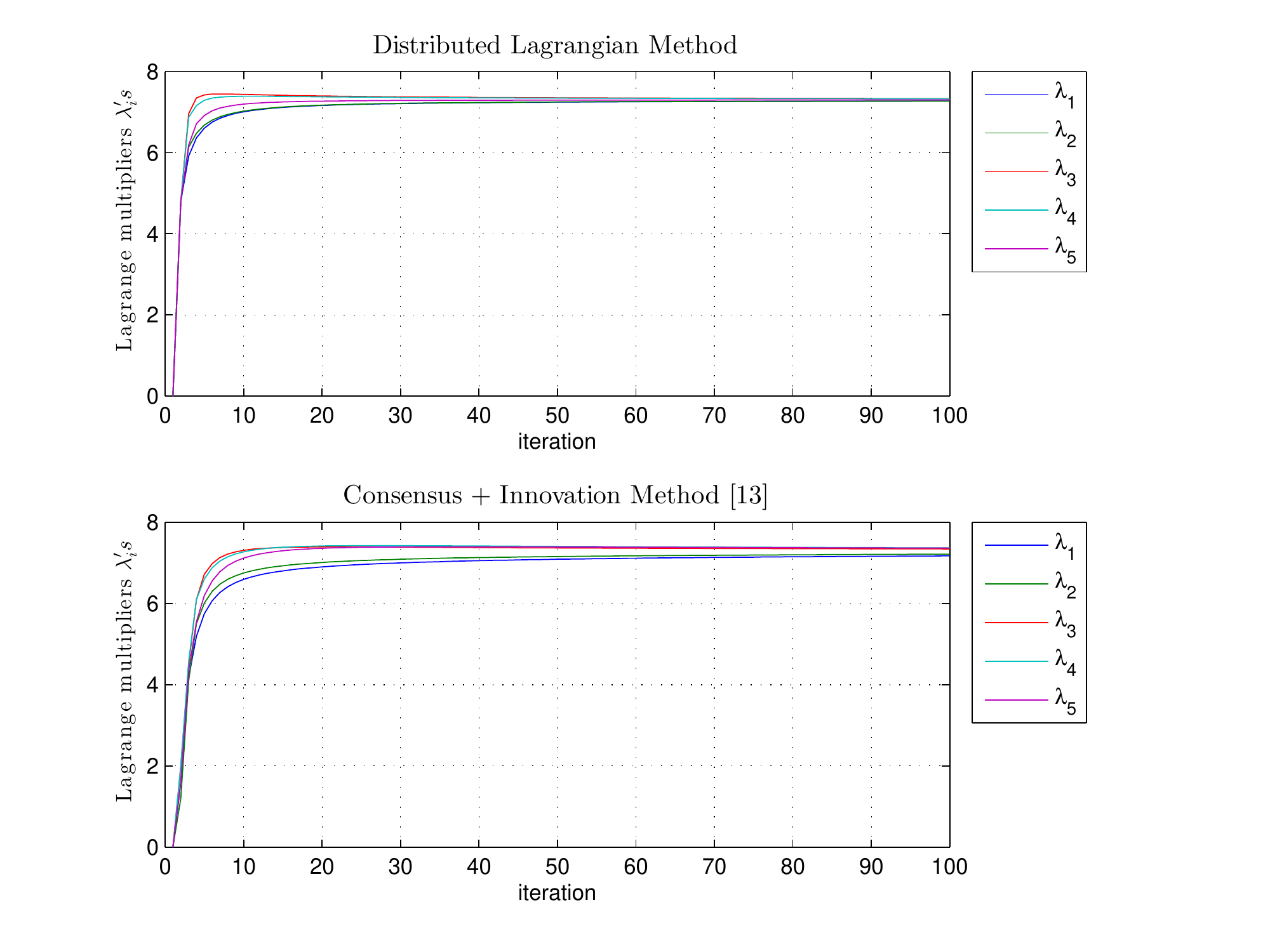}
\vspace{-0.8cm}
\caption{Economic dispatch problem on IEEE-14 bus systems.}\label{fig:IEEE14Bus_Simulation}\vspace{-0.5cm}
\end{figure}

We first study the economic dispatch problem on IEEE-14 bus systems, where generators are located at buses $1,2,3,6,$ and $8$ as shown in Fig. \ref{fig:IEEE14Bus}. Each generator $i$ suffers a quadratic cost as a function of the amount of its generated power $P_i$, i.e., $f_i(P_i) = \gamma_iP_i^2+\beta_iP_i$ where $\gamma_i,\beta_i$ are cost coefficients of generator $i$. We also assume that each generator $i$ can only generate a limited amount of power represented by $[0,P_i^{\text{max}}]$. The coefficients of each generator are listed in Table \ref{tab:14BusGeneratorPar} which are adopted from \cite{Kar2012}. The total load demand required from the system is $P = 300 MW$, which is a typical load value chosen for the purpose of this simulation. The goal now is to meet the load demand while minimizing the total cost incurred at the generators. Here we assume the connection between generators is represented by an undirected circle graph. 

Since the local function at each generator is quadratic, the updates in the DLM algorithm can be simplified as 
\begin{align}
&x_i(k+1) =  \mathcal{P}_{[0,P_i^{\max}]}\left[\begin{array}{c}
\frac{-\sum_{j\in \mathcal{N}_i}a_{ij}\lambda_j(k)}{2\gamma_i}-\frac{\beta_i}{2\gamma_i}
\end{array}
\right]\nonumber \\
&\lambda_i(k+1) = \sum_{j\in \mathcal{N}_i}a_{ij}\lambda_j(k) - \alpha(k)(x_i(k+1)-b_i),\nonumber
\end{align}
where $\mathcal{P}_{[0,P_i^{\max}]}[x]$ is the projection of $x$ on to the interval $[0,P_i^{\max}]$. 

In this study, we compare the performance of our DLM algorithm with the consensus + innovation method (CIM) studied in \cite{Kar2012}. We select the same step size as that used in \cite{Kar2012}, i.e., $\alpha(k) = 0.08/k^{0.85}$, which also satisfies the conditions in Lemma \ref{CLlem:DualConvergence}. The simulations of these two methods when solving the economic dispatch problem on the IEEE-14 bus test system are shown in Figure  \ref{fig:IEEE14Bus_Simulation}. In each figure, our method is shown on the top plot while the CIM algorithm of \cite{Kar2012} is shown in the bottom. The plots show that in this study DLM outperforms CIM. In particular, DLM schedules the power to the generators after $20$ iterations while CMI requires $40$ iterations. Moreover, the Lagrange multipliers $\lambda_i$ in DLM converge after $60$ iterations, while those in CIM have not converged even after $100$ iterations.

\begin{figure}
\centering
\includegraphics[width=3.5in,height=2.5in,keepaspectratio]{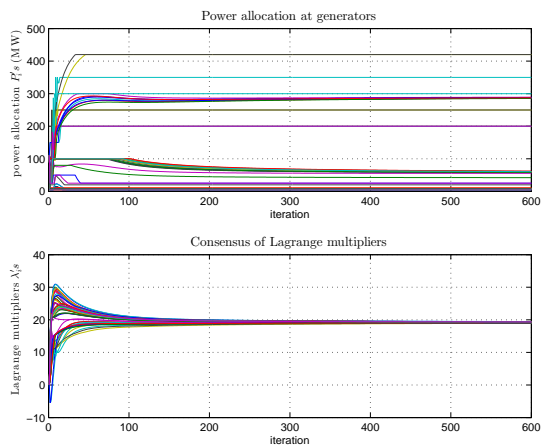}
\vspace{-0.3cm}
\caption{Economic dispatch problem on IEEE-118 bus systems.}\label{fig:IEEE118_ED}
\end{figure}

\subsection{Case study 2: Economic dispatch on IEEE-118 bus systems}
We now consider the economic dispatch problem on a larger system, the IEEE-118 bus test system \cite{IEEE118BusWebsite}. This system has $54$ generators connected by bus lines. Each generator $i$ suffers a quadratic cost as a function of the amount of its generated power $P_i$, i.e., $f_i(P_i) = \mu_i + \beta_iP_i + \gamma_i P_i^2$. The coefficients of functions $f_i$ belong to the ranges $\mu_i\in[6.78,74.33]$, $\beta_i\in[8.3391,37.6968]$, and $\gamma_i\in[0.0024,0.0697]$. The units of $\mu_i,\beta_i,\gamma_i$ are $MBtu, MBtu/MW$ and $MBtu/MW^2$, respectively. Each $P_i$ is constrained on some interval $[P_i^{\text{min}},P_i^{\max}]$ where these values vary as $P_i^{\text{min}}\in[5,150]$ and $P_i^{\text{max}}\in[150,400]$. The unit of power in this system is $MW$. The total load demand required from the system is $P = 6000 (MW)$, which again is a typical load demand chosen for the purpose of this simulation. In this study we select the step size $\alpha(k) = 1/k$ for $k\geq 1$ with $\alpha(0) = 1$. Moreover, the connections between generators are based on the physical connections given by bus data \cite{IEEE118BusWebsite}. The simulation results are given in Figure \ref{fig:IEEE118_ED} which demonstrate the convergence of our method for this system. 

\vspace{-0.2cm}
\section{Concluding Remarks}\label{sec:Conclusion}
In this paper, we have proposed a distributed Lagrangian method for network resource allocation problems. Under a minimal assumption on the convexity of the nodes' functions, we show that our method achieves asymptotic convergence to the optimal value with the rate, ${\mathcal{O}(n\ln(k)/(1-\sigma_2)\sqrt{k})}$. Finally, to show the effectiveness of our method, we applied our method to solve the economic dispatch problem on the benchmark IEEE-14 and IEEE-118 bus test systems. 

Our main motivation comes from applications in power systems, where our method provides a useful tool for determining near optimal power allocation in a distributed power system. One direction following this work is to couple our method with other control tasks in power systems, for example, frequency control problems \cite{StevenLow2014,Dorfler2016}. A second interesting open direction is to consider the case of unknown or time-varying resources as is typically found in actual power systems where the demand is changing over time and the relevant data may be uncertain.    

\section*{Acknowledgement}
The authors would like to thank Alex Olshevsky, Rafal Goebel, and Daniel Liberzon for their insightful comments which helped to improve the quality of this paper.

\bibliographystyle{plain}
\bibliography{refs}

\end{document}